\documentclass[12pt]{amsart}
\usepackage{amsthm,amsmath,amsfonts,amssymb,latexsym,amscd,mathrsfs,hyperref,cleveref,graphicx}
\usepackage{setspace}
\usepackage{cite}
\usepackage{phonetic}
\setdisplayskipstretch{2}
\usepackage{cite}
\usepackage{crossreftools}

\makeatletter
\renewcommand{\pod}[1]{\allowbreak\mathchoice
  {\if@display \mkern 18mu\else \mkern 8mu\fi (#1)}
  {\if@display \mkern 18mu\else \mkern 8mu\fi (#1)}
  {\mkern4mu(#1)}
  {\mkern4mu(#1)}
}

\makeatletter
\@namedef{subjclassname@2020}{%
  \textup{2020} Mathematics Subject Classification}
\makeatother

\theoremstyle{plain}

\newtheorem{lem}{Lemma}

\newtheorem*{thm*}{Theorem}
\newtheorem*{cor*}{Corollary}
\newtheorem*{prop*}{Proposition}

\theoremstyle{definition}

\newtheorem*{exa*}{Example}

\Crefname{thm}{Theorem}{theorems}
\Crefname{lem}{Lemma}{lemmas}

\newcommand{\bs}{\boldsymbol}

\def \a{\alpha} \def \b{\beta}  \def \e{\varepsilon} \def \g{\gamma}  \def \l{\lambda}  \def \t{\theta} 

\makeatletter
\def\widebreve{\mathpalette\wide@breve}
\def\wide@breve#1#2{\sbox\z@{$#1#2$}%
     \mathop{\vbox{\m@th\ialign{##\crcr
\kern0.08em\brevefill#1{0.8\wd\z@}\crcr\noalign{\nointerlineskip}%
                    $\hss#1#2\hss$\crcr}}}\limits}
\def\brevefill#1#2{$\m@th\sbox\tw@{$#1($}%
  \hss\resizebox{#2}{\wd\tw@}{\rotatebox[origin=c]{90}{\upshape(}}\hss$}
\makeatletter

\numberwithin{equation}{section}

\renewcommand{\labelenumi}{\setlength{\labelwidth}{\leftmargin}
   \addtolength{\labelwidth}{-\labelsep}
   \hbox to \labelwidth{\theenumi.\hfill}}

\begin{document}
\title{$L^p$ maximal estimates for quadratic Weyl sums}
\author{Roger Baker}
\address{Department of Mathematics\newline
\indent Brigham Young University\newline
\indent Provo, UT 84602, U.S.A}
\email{baker@math.byu.edu}

 \begin{abstract}
Let $p \ge 1$. We give upper and lower bounds for
 \[M_p(N): = \Bigg\|\sup_{0 \le t \le 1} \Bigg|
 \sum_{n=1}^N e(nx + n^2t)\Bigg|\, \Bigg\|_{L^p[0,1]}^p\]
that are of the same order of magnitude.
 \end{abstract}

\keywords{Weyl sum, $L^p$ maximal estimate.}

\subjclass[2020]{Primary 11L15; Secondary 35Q55}
\maketitle

 \section{Introduction}\label{sec:intro}

Let $e(t) = e^{2\pi it}$. The Schr\"odinger equation for $u(x,t)$,
 \[\frac i{2\pi} \ \frac{\partial u}{\partial t} -
 \frac 1{4\pi^2} \ \frac{\partial^2u}{\partial x^2} = 0\]
with periodic boundary condition
 \[u(x,0) = \sum_k a_k e(kx)\]
has the solution
 \[u(x,t) = \sum_k a_k e(kx + k^2t).\]
The associated maximal function
 \[\sup_t |u(x,t)|\]
was studied by Carleson \cite{carl} for solutions of the free Schr\"odinger equation on the real line; see Pierce \cite{pie} and Moyau and Vega \cite{moyveg} for further results and references. In the periodic case, Moyau and Vega showed that for every $\e > 0$,
 \newpage
 
 \begin{align}
\Bigg\|\sup_{0 \le t \le 1}  \Bigg| \sum_{n=1}^N &a_n e(nx + n^2t)\Bigg|\ \Bigg\|_{L^p[0,1]}\label{eq1.1}\\
&\qquad \le C(\e) N^{1/3 + \e} \|a\|_2 \quad (1 \le p \le 6),\notag
 \end{align}
where $\|a\|_2 = \left(\sum\limits_{n=1}^N |a_n|^2\right)^{1/2}$. When $p=6$, this is sharp up to the loss of $N^\e$ (let $a_n = 1$ and restrict to $[0, 10^{-6} N^{-1}]$). However, in a recent paper Barron \cite{bar} improved \eqref{eq1.1} for $1 \le p\le 4$ in the special case where $a_n=1$. Let
 \[M_p(N) = \int_0^1 \max_{0 \le t \le 1} \Bigg|
 \sum_{n=1}^N e(nx + n^2t)\Bigg|^p dx.\]
Barron showed that $M_4(N) \ll_\e N^{3+\e}$, improving \eqref{eq1.1} in this case.

In the present paper we sharpen Barron's result and determine the order of magnitude of $M_p(N)$ for $p \ge 1$.

 \begin{thm*}
We have
 \begin{equation}\label{eq1.2}
N^{a(p)}(\log N)^{b(p)} \ll M_p(N) \ll N^{a(p)}(\log N)^{b(p)}, 
 \end{equation}
where
 \begin{align*}
 a(p) &= \begin{cases}
 \frac{3p}4 & (1 \le p \le 4)\\[1mm]
 p - 1 & (p > 4)\end{cases}\\
\intertext{and}
b(p) &= \begin{cases}
1 & (p = 4)\\[1mm]
0 & (p \ge 1, p \ne 4).\end{cases}
 \end{align*}
The implied constants depend at most on $p$.
 \end{thm*}

After collecting some lemmas in Section \ref{sec:lemsproofthm}, the upper bound in \eqref{eq1.2} is proved in Section \ref{sec:proofupbound1.2}. For $p > 4$, the lower bound is essentially trivial (as above, restrict integration to $[0, 10^{-6}N^{-1}]$). The lower bound for $1 \le p \le 4$ is proved in Section \ref{sec:prooflowbound12}.

We indicate briefly how these proofs proceed. Let
 \[S(x,t) = \sum_{n=1}^N e(nx + n^2t).\]
We define $t(x)$ to be the least number in $[0,1]$ satisfying
 \[|S(x,t(x))| = \max_{0 \le t \le 1} \Bigg|
 \sum_{n=1}^N e(nx + n^2)\Bigg|.\]
For the upper bound in \eqref{eq1.2}, we may restrict integration to the set where $|S(x,t(x))|$ is large. For $x$ in this set, we approximate $|S(x,t(x))|$ by
 \begin{equation}\label{eq1.3}
q^{-1} S(q,\bs a) I(\bs \b),
 \end{equation}
where $\frac{\bs a}q = \left(\frac{a_1}q, \frac{a_2}q\right)$ is close to $(x,t(x))$, and $\bs\b = (\b_1,\b_2)$ is given by
 \[\bs\b = \left(x - \frac{a_1}q, t(x) - \frac{a_2}q\right).\]
Here
 \[S(q,\bs a) = \sum_{y=1}^q e \left(
 \frac{ya_1 + y^2a_2}q\right) \, , \, I(\bs \b) =
 \int_0^N e(\b_1\g + \b_2\g^2)d\g.\]
We can now use standard upper bounds for $S(q,\bs a)$ and $I(\bs \b)$ to complete the proof. The existence of $\frac{\bs a}q$ with small enough $\b_1$ and $\b_2$, which `jump-starts' the argument, depends on the work of Baker \cite{rcb}.

For the lower bound in \eqref{eq1.2}, we restrict integration to
 \begin{equation}\label{eq1.4}
E = \bigcup_{\substack{3 \le q \le \frac 16\, N^{1/2}\\
q \text{ odd}}} \ \bigcup_{\substack{ 1 \le a_1 \le q\\
(a_1,q)=1}} \left[\frac{a_1}q + \frac{15}N, \frac{a_1}q + \frac{16}N\right].
 \end{equation}
We define a suitable $v(x)$ $(x\in E)$ for which $S(x,v(x))$ is sufficiently closely approximated by \eqref{eq1.3} with $a_2 = 1$, and now
 \begin{equation}\label{eq1.5}
\bs \b = \left(x - \frac{a_1}q \, , \, v(x) - \frac 1q\right).
 \end{equation}
We find that
 \begin{equation}\label{eq1.6}
|I(\bs\b)| \gg N.
 \end{equation}
We also have
 \begin{equation}\label{eq1.7}
|S(q, a_1,1)| \gg q^{1/2} 
 \end{equation}
because $q$ is odd. Thus the quantity in \eqref{eq1.3} is $\gg Nq^{-1/2}$, leading quickly to the lower bound in \eqref{eq1.2}. 
 \bigskip

 \section{Lemmas for the proof of the theorem.}\label{sec:lemsproofthm}

 \begin{lem}\label{lem:suppose|S(x,t)|}
Suppose that
 \[|S(x,t)| \ge P \ge N^{1/2+\e}.\]
Then there exist a natural number $q$ and integers $a_1$, $a_2$ such that
 \begin{gather}
q \ll (NP^{-1})^2 N^\e \ , \ (q, a_1, a_2) = 1,\label{eq2.1}\\[2mm]
|qx - a_1| \ll (NP^{-1})^2 N^{\e-1} \ , \ |qt - a_2| \ll (NP^{-1})^2 N^{\e-2}.\label{eq2.2}
 \end{gather} 
 \end{lem}
 
 \begin{proof}
This is a special case of the main result of \cite{rcb}.
 \end{proof}
 
 \begin{lem}\label{lem:letx=fraca1q+b1}
Let $x = \frac{a_1}q + \b_1$ and $t = \frac{a_2}q + \b_2$. Then
 \[S(x,t) = q^{-1}S(q,\bs a) I(\bs \b) + \Delta\]
where
 \[\Delta \ll q(1 + |\b_1| N + |\b_2|N^2).\]
 \end{lem}
 
 \begin{proof}
This is a special case of Vaughan \cite{vau}, Theorem 7.2.
 \end{proof}
 
 \begin{lem}\label{lem:(i)(ii)}
(i) For $q \ge 1$, $(q, a_1, a_2) = 1$, we have
 \[S(q, \bs a) \ll q^{1/2}.\]

(ii) Let $q$ be odd and $(q, a_1) =1$. Let $[4,q]$ denote the inverse of $4\pmod q$. Then
 \[S(q,a_1, 1) = \begin{cases}
 e\left(-\frac{[4,q]a_1^2}q\right) q^{1/2} & (q \equiv 1 \pmod 4)\\[4mm]
 e\left(-\frac{[4,q]a_1^2}q\right) iq^{1/2} & (q \equiv 3 \pmod 4).\end{cases}\]
 \end{lem}
 
 \begin{proof}
Part (i) is a consequence of Lemmas 3--9 of Estermann \cite{est}. Part (ii) follows from Lemma 3 of \cite{est} in conjunction with the evaluation of the sum $S$ in Chapter 2 of Davenport \cite{dav}.
 \end{proof}
 
 \begin{lem}\label{lem:letA,Xposnums}
Let $A$, $X$ be positive numbers. Then
 \[\int_{-X}^X e(A\g^2)d\g = \frac 1{2\sqrt A}
 + \frac \l{AX},\]
where $|\l| \le 7/12$.
 \end{lem}
 
 \begin{proof}
Let $Y > 0$. The integral
 \[\int_C e(z^2)dz\]
is 0, where $C$ consists of straight line segments $[0,Y]$ and $\left[Y e\left(\frac 18\right),0\right]$, together with the circular arc $C_1$ from $Y$ to $Ye(1/8)$. Hence,
 \begin{align*}
\int_0^Y e(x^2)dx + \int_{C_1} e(z^2)dz &= e\left(\frac 18\right) \int_0^Y e^{-2\pi x^2}dx\\[2mm]
&= \frac{e\left(\frac 18\right)}{2\sqrt 2} + \nu \int_Y^\infty \frac{dx}{2\pi x^2} = \frac{e\left(\frac 18\right)}{2\sqrt 2} + \frac \nu{2\pi Y}
 \end{align*}
where $|\nu| \le 1$. The integral over $C_1$ has absolute value
 \begin{align*}
&\le Y \int_0^{\pi/4} \exp(-2\pi Y^2 \sin 2\t) d\t\\[2mm]
&\le Y \int_0^\infty \exp(-8Y^2\t)d\t = \frac 1{8Y}.
 \end{align*}
Hence
 \[\int_{-Y}^Y e(x^2)dx = 
 \frac{e\left(\frac 18\right)+ e\left(-\frac 18\right)}{2\sqrt 2}
 + \frac{\l_1}Y = \frac 12 + \frac{\l_1}Y\]
where $|\l_1| \le \frac 14 + \frac 1\pi < \frac 7{12}$. The general case follows on a change of variable.
 \end{proof}
 
 \begin{lem}\label{lem:I(bsbeta<<N}
We have
 \[I(\bs \b) \ll N(1 + |\b_1|N + |\b_2|N^2)^{-1/2}.\]
 \end{lem}
 
 \begin{proof}
This is a special case of Vaughan \cite{vau}, Theorem 7.3. 
 \end{proof}

 \begin{lem}\label{lem:letYge2}
(i) Let $Y \ge 2$. We have
 \[\sum_{\substack{y \le Y\\
 y \text{ odd}}} \phi(y) = \frac 2{\pi^2}\, Y^2 +
 O(Y\log Y).\]

(ii) For large $Y$, we have 
 \[\sum_{\substack{
\frac Y2 < y \le Y\\
y \text{ odd}}} \, \frac{\phi(y)}{y^\a} 
\gg_\a Y^{2-\a} \quad (\a \text{ real}).\]

(iii) For large $Y$, we have
 \[\sum_{\substack{
 y\le Y\\
 y \text{ odd}}} \, \frac{\phi(y)}{y^2} \gg \log Y.\]
 \end{lem}
 
 \begin{proof}
(i) We have
 \begin{align*}
\sum_{\substack{y \le Y\\
y \text{ odd}}} \phi(y) &= \sum_{\substack{
y \le Y\\
y \text{ odd}}} y \sum_{d\mid y} \ \frac{\mu(d)}d\\[4mm]
&= \sum_{\substack{
d \le Y\\
d \text{ odd}}} \mu(d) \sum_{\substack{
n \le Y/d\\
n \text{ odd}}} n\\[4mm]
&= \sum_{\substack{
d\le Y\\
d \text{ odd}}} \mu(d) \left(\frac{Y^2}{4d^2} + O
\left(\frac Yd\right)\right).
 \end{align*}
Since
 \[\sum_{\substack{
 d = 1\\
 d \text{ odd}}}^\infty \ \frac{\mu(d)}{d^2} =
 \frac 1{1 - \frac 1{2^2}} \ \prod_p \left(1 -
 \frac 1{p^2}\right)= \frac 43 \cdot \frac 6{\pi^2}
 = \frac 8{\pi^2}\, ,\]
the result follows easily.

Now (ii) is an obvious consequence of (i). Next,
 \[\sum_{\substack{
 2^{j-1} < y \le 2^j\\
 y \text{ odd}}} \ \frac{\phi(y)}{y^2} \gg 1\]
from (ii). Summing over $\gg \log Y$ values of $j$ yields (iii).
 \end{proof}
 \bigskip
  
 \section{Proof of the upper bound in \eqref{eq1.2}.}\label{sec:proofupbound1.2}

Let
 \[F = \{x \in[0,1] \cdot |S(x,t(x))| > N^{3/4}\}.\]
Then
 \[\int_0^1 |S(x,t(x))|^p dx \le N^{3p/4} \ll
 N^{a(p)} \quad (p \ge 1).\]
Thus we may confine integration of $|S(x,t(x))|^p$ to the set $F$.

Let $x \in F$. By Lemma \ref{lem:suppose|S(x,t)|} there exist integers $q$, $a_1$, $a_2$ with $(q, a_1, a_2) = 1$,
 \begin{align}
1 \le q \ll N^{1/2 + \e} \ , \ &\b_1 := x - \frac{a_1}q \ll q^{-1} N^{-1/2+\e},\label{eq3.1}\\
&\hskip .75in \b_2 := t(x) - \frac{a_2}q \ll q^{-1}N^{-3/2 +\e}.\notag
 \end{align}

By Lemma \ref{lem:letx=fraca1q+b1},
 \[S(x,t(x)) = q^{-1} S(q, \bs a) I(\bs\b) + \Delta\]
where
 \[\Delta \ll N^{1/2+\e}.\]
It follows at once that
 \begin{align*}
N^{3/4} &\le |S(x,t(x))| \ll q^{-1} |S(q,\bs a) I(\bs\b)|\\[2mm]
&\ll q^{-1/2} \min(N, N^{1/2} |\b_1|^{-1/2}). 
 \end{align*}
Sharpening \eqref{eq3.1} a little, we find that
 \[q \ll N^{1/2} \ , \ \b_1 \ll q^{-1} N^{1/2}.\]
Hence, for a suitable constant $C > 0$,
 \begin{equation}\label{eq3.2}
\int_F |S(x,t(x))|^p dx \ll \sum_{q \le CN^{1/2}} \ \sum_{a_1 =0}^q q^{-p/2} \int_{-Cq^{-1}N^{-1/2}}^{Cq^{-1}N^{-1/2}} \min(N^p, N^{p/2}|\b_1|^{-p/2})d\b_1.
 \end{equation}
At this stage we exclude $p=2$ from the argument. The contribution to the integral on the right-hand side of \eqref{eq3.2} from $|\b_1| \le N^{-1}$ is $\ll N^{\b - 1}$. From the two remaining intervals we obtain a contribution
 \[\ll \begin{cases}
 N^{p/2}(N^{-1})^{1-\frac p2} = N^{p-1} & \text{if } p > 2\\[2mm]
 N^{p/2}(q^{-1}N^{-1/2})^{1-p/2} = N^{\frac{3p}4-\frac 12} q^{\frac p2 - 1} & \text{if } p < 2. \end{cases}\]
We conclude that
 \begin{equation}\label{eq3.3}
\int_F |S(x,t(x))|^p dx \ll \sum_{q \le CN^{1/2}} q^{1-p/2} (N^{p-1} + N^{3p/4 - 1/2} q^{\frac p2 - 1}).
 \end{equation}

For $p < 4$, $p \ne 2$,
 \begin{equation}\label{eq3.4}
\sum_{q \le CN^{1/2}} q^{1-p/2} \ll (N^{\frac 12})^{2-\frac p2} = N^{1-\frac p4}. 
 \end{equation}
For $p=4$,
 \begin{equation}\label{eq3.5}
\sum_{q \le CN^{1/2}} q^{1-p/2} \ll \log N.
 \end{equation}
For $p > 4$,
 \begin{equation}\label{eq3.6}
\sum_{q \le CN^{1/2}} q^{1-p/2} \ll 1.
 \end{equation}

The desired upper bound follows on combining \eqref{eq3.3}--\eqref{eq3.6}, except for $p=2$. Since $\|f\|_{L^p[0,1]}$ is increasing with $p$ for a given function $f$ on $[0,1]$, we obtained the desired result for $p=2$ by comparison with the result for $p=3$.

 \section{Proof of the lower bound in \eqref{eq1.2} for $1 \le p \le 4$.}\label{sec:prooflowbound12}
 
Let
 \[J(q,a_1) = \left[\frac{a_1}q + \frac{15}N \ , \
 \frac{a_1}q + \frac{16}N\right]\]
be one of the intervals in the definition of $E$. The intervals $J(q_1, a_1)$, $J(q_2, a_2)$ of $E$ with $\frac{a_1}{q_1} \ne \frac{a_2}{q_2}$ are non-overlapping, since
 \[\left|\frac{a_1}{q_1} - \frac{a_2}{q_2}\right| \ge
 \frac 1{q_1q_2} \ge \frac{36}N.\]

For $x \in J(q,a_1)$, we write $a_2 = 1$, $\b_1 = x - \frac{a_1}q$, $\b_2 = -\frac{\b_1}N$,
 \begin{equation}\label{eq4.1}
v(x) = \frac{a_2}q - \frac{\b_1}N = \frac 1q - \frac{\b_1}N.
 \end{equation}

We have

 \begin{equation}\label{eq4.2}
M_p(N) \ge \sum_{\substack{
3 \le q \le N^{1/2}/6\\
q \text{ odd}}} \ \sum_{\substack{
a_1 = 1\\
(a_1,q) = 1}}^q \int_{J(q,a_1)} |S(x,v(x))|^p dx
 \end{equation}
We shall show that for $x$ in any of the intervals $J(q, a_1)$ in \eqref{eq4.2}, we have
 \begin{equation}\label{eq4.3}
|S(x,v(x))| \gg Nq^{-1/2}.
 \end{equation}
The lower bound in \eqref{eq1.2} follows from \eqref{eq4.2}, \eqref{eq4.3} on an application of Lemma \ref{lem:letYge2} (ii), (iii):
 \begin{align*}
M_p(N) &\gg \sum_{\substack{
3 \le q \le N^{1/2}/6\\
q \text{ odd}}} \ \sum_{\substack{
a_1 = 1\\
(a_1,q) = 1}}^q \int_{J(q,a)} q^{-p/2} N^pd\b_1\\[4mm]
&\gg \sum_{\substack{
3 \le q \le N^{1/2}/6\\
q \text{ odd}}} \phi(q)q^{-p/2} N^{p-1}\\[4mm]
&\gg N^{p-1}(N^{1/2})^{2-p/2} (\log N)^{b(p)} = N^{3p/4} (\log N)^{b(p)}.
 \end{align*}
Thus it remains to prove \eqref{eq3.2}. With $x \in J(q, a_1)$ and $v(x)$ as in \eqref{eq4.1}, we have
 \begin{equation}\label{eq4.4}
|S(x,v(x))| = q^{-1} S(q,a,1) I(\b_1, \b_2) + \Delta
 \end{equation}
where
 \begin{equation}\label{eq4.5}
\Delta \ll q(1 + |\b_1| N + |\b_2| N^2) \ll N^{1/2}.
 \end{equation}
Now

 \begin{align*}
I(\b_1,\b_2) &= \int_0^N e\left(\b_1\g - \frac{\b_1}N \, \g^2\right) d\g\\[4mm]
&= \int_0^N e\left(-\frac{\b_1}N \, \left(\g - \frac N2\right)^2\right) e\left(\frac{\b_1N}4\right) d\g,\\[4mm]
|I(\b_1, \b_2)| &= \left|\int_{-N/2}^{N/2} e\left(\frac{\b_1\g^2}N\right) d\g\right|.
 \end{align*}
Now we apply Lemma \ref{lem:letA,Xposnums} with $X = N/2$, $A = \b_1/N$ to obtain
 \[\int_{-N/2}^{N/2} e\left(\frac{\b_1\g^2}N\right)d\g
 = \frac{N^{1/2}}{2\b_1^{1/2}} + \frac{2\l}{\b_1},
 |\l| \le \frac 7{12}.\]
Since $\b_1 \ge \frac{15}N$, it is easily seen that
 \[\frac{N^{1/2}}{2\b_1^{1/2}} > \frac{3\l}{\b_1}\]
and so
 \[|I(\b_1, \b_2)| \gg N.\]
Recalling Lemma \ref{lem:(i)(ii)} (ii), we have
 \[|q^{-1} S(q,a_1) I(\b_1,\b_2)| \gg
 \frac N{q^{1/2}} \gg N^{3/4}.\]
Combining this with \eqref{eq4.4}, \eqref{eq4.5} we conclude that \eqref{eq4.3} holds in $J(q,a_1)$. This completes the proof of the lower bound in \eqref{eq1.2}.

 \end{document}